\documentclass[10pt]{article}
\usepackage{latexsym,amsmath}
\usepackage{graphicx}
\usepackage{caption}
\usepackage{amsfonts}
\usepackage{isomath}
\usepackage{amssymb,array}
\usepackage{epsfig}
\usepackage{color}
\usepackage{setspace}
\usepackage{amsthm}
\usepackage{mathrsfs}
\usepackage[margin=1in]{geometry}
\newtheorem{thm}{Theorem}[section]
\newtheorem{lem}{Lemma}[section]
\newtheorem{re}{Remark}[section]
\newtheorem{de}{Definition}[section]
\newtheorem{p}{Proposition}[section]
\newtheorem{ex}{Example}[section]
\usepackage{apacite}

\begin{document}
	
\title{Generalized Weighted Survival and Failure Entropies and their Dynamic Versions}
\author{Siddhartha Chakraborty, Biswabrata Pradhan}
\maketitle
\begin{abstract}
 The weighted forms of generalized survival and failure entropies of order ($\alpha,\beta$) are proposed and some properties are obtained. We further propose the dynamic versions of weighted generalized survival and failures entropies and obtained some properties and bounds. Characterization for Rayleigh and power distributions are done by dynamic weighted generalized entropies. We further consider the empirical versions of generalized weighted survival and failure entropies and using the difference between theoretical and empirical survival entropies a test for exponentiality is considered.
	
\end{abstract}

\section{Introduction} \citeA{shannon1948mathematical} introduced the concept of differential entropy and since then it has been playing an improtant role in the field of information theory, thermodynamics, statistical mechanics and reliability. Let $X$ be a non-negative absolutely continuous random variable (rv) having cumulative distribution function (cdf) $F(x)$ and probability density function (pdf) $f(x)$, Then Shannon entropy of $X$ is given by 
\begin{align}\label{e1}
H(X)=-\int_{0}^{\infty}f(x)\text{log}f(x)dx.
\end{align}
There are various generalizations of Shannon entropy considered by many authors. Two most important ones are due to \citeA{renyi1961measures} and \citeA{varma1966generalizations}.
 Reyni's entropy of $X$ is given by $$H_{\alpha}(X)=\frac{1}{1-\alpha}\text{log}\int_{0}^{\infty}f^{\alpha}(x)dx,\; \alpha(\neq 1)>0$$ and Verma's entropy of X is defined as 
$$H_{\alpha,\beta}(X)=\frac{1}{\beta-\alpha}\text{log}\int_{0}^{\infty}f^{\alpha+\beta-1}(x)dx,\; \beta\geq 1,\; \beta-1\;<\alpha<\;\beta.$$
When $\alpha\to 1$, $H_{\alpha}(X)\to H(X)$. For $\beta=1$, $H_{\alpha,\beta}(X)$ reduces to $H_{\alpha}(X)$ and when both $\alpha,\beta$ tends to 1, $H_{\alpha,\beta}(X)$ tends to $H(X)$. \\

If an item has survived time $t$ then in order to incorporate the residual lifetime of the item, \citeA{ebrahimi1996measure} proposed dynamic entropy as $H(X;t)=-\int_{t}^{\infty}\frac{f(x)}{\bar{F}(t)}\text{log}\frac{f(x)}{\bar{F}(t)}dx$, where $\bar{F}(x)=1-F(x)$ is the survival function (sf) of $X$. \citeA{di2002entropy} proposed the concept of dynamic past entropy measure as $\bar{H}(X;t)=-\int_{0}^{t}\frac{f(x)}{F(t)}\text{log}\frac{f(x)}{F(t)}dx$.\\

Recently \shortciteA{rao2004cumulative} and \shortciteA{rao2005more} have proposed cumulative residual entropy measure as \\
$\epsilon(X)=-\int_{0}^{\infty}\bar{F}(x)\text{log}\bar{F}(x)dx$. It may be noted that $\epsilon(X)$ measures the uncertainty when cdf exists but pdf does not. \citeA{asadi2007dynamic} proposed the dynamic form of $\epsilon(X)$ and \citeA{di2009cumulative} proposed cumulative entropy $\bar{\epsilon}(X)=-\int_{0}^{\infty}F(x)\text{log}F(x)dx$. \cite{zografos2005survival} proposed survival entropy of order $\alpha$ as $\xi_{\alpha}(X)=\frac{1}{1-\alpha}\text{log}\int_{0}^{\infty}\bar{F}^{\alpha}(x)dx,\;\alpha(\neq1)>0$ and \shortciteA{abbasnejad2010dynamic} obtained its dynamic version. \citeA{abbasnejad2011some} introduced the failure entropy of order $\alpha$ as $f\xi_{\alpha}(X)=\frac{1}{1-\alpha}\text{log}\int_{0}^{\infty}F^{\alpha}(x)dx$ and also obtained its dynamic version. \\

Motivated from \citeA{zografos2005survival}, \shortciteA{abbasnejad2010dynamic} and \shortciteA{abbasnejad2011some}, \citeA{kayal2015generalized} proposed generalized survival and failure entropies of order $(\alpha,\beta)$ as 
\begin{align}\label{e2}
\xi_{\alpha,\beta}(X)=\frac{1}{\beta-\alpha}\text{log}\int_{0}^{\infty}\bar{F}^{\alpha+\beta-1}(x)dx,\; \beta\geq 1,\; \beta-1\;<\alpha<\;\beta
\end{align}
and 
\begin{align}\label{e3}
	f\xi_{\alpha,\beta}(X)=\frac{1}{\beta-\alpha}\text{log}\int_{0}^{\infty}{F}^{\alpha+\beta-1}(x)dx,\; \beta\geq 1,\; \beta-1\;<\alpha<\;\beta.
\end{align}
They also considered their dynamic versions and obtained characterization results for exponential, pareto and power distributions.\\
All these above measures are shift independent and gives equal weights to the occurance of events. But in practical situations such as communication theory and Reliability, a shift dependent measure is often required. To incorporate this issue, \citeA{belis1968quantitative} have introduced the concept of weighted entropy as $H^w(X)=-\int_{0}^{\infty}xf(x)\text{log}xdx$. Since then, several works have been done on weighted entropies. One may refer to \shortciteA{misagh2011weighted}, \shortciteA{mirali2017weighted}, \shortciteA{mirali2017dynamic,mirali2017some}, \shortciteA{rajesh2017dynamic}, \shortciteA{nair2017study}, \citeA{khammar2018weighted}, \citeA{das2017weighted} and \shortciteA{nourbakhsh2016weighted}, for details on weightes entriopy measures.\\

In this article, we propose generalized weighted survival and failure entropies of order $(\theta_1,\theta_2)$ and their dynamic versions. The properties of the proposed entropy measures are discussed. The rest of the paper is organized as follows. In section 2, we introduce generalized weighted survival entropy and obtain its properties. The dynamic versions of generalized weighted survival entropy is discussed in section 3. Characterization results for Rayleigh distribution are obtained using generalized dynamic weighted survival entropy in section 4. We propose generalized weighted failure entropy and dynamic failure entropy in section 5. Characterization results for power distribution are obtained based on generalized dynamic weighted failure entropy. We obtain some inequalities and bounds for the proposed entropy measures in section 6. The empirical generalized weighted survival and failure entropies are provided in section 7. A goodness-of-fit test for exponential distribution is discussed in section 8. Finally, we conclude the paper in section 9. 

\section{Generalized weighted survival entropy of order $(\alpha,\beta)$}
Here we introduce generalized weighted survival entropy and obtain some properties. 

\begin{de}
	Generalized weighted survival entropy (GWSE) of order $(\alpha,\beta)$ is proposed as 
	\begin{align}\label{e4}
	\xi_{\alpha,\beta}^w(X)=\frac{1}{\beta-\alpha}\log\int_{0}^{\infty}x\bar{F}^{\alpha+\beta-1}(x)dx,\; \beta\geq 1,\; \beta-1\;<\alpha<\;\beta.
	\end{align}
\end{de}
To illustrate the usefulness of the proposed entropy measure, we consider the following example.
\begin{ex}
	Suppose $X$ and $Y$ have pdfs $f(x)=\frac{1}{b-a},\;a<x<b$ and $g(y)=\frac{1}{b-a},\;a+h<y<b+h,\;h>0$, respectively. From (\ref{e2}), we have $\xi_{\alpha,\beta}(X)=\xi_{\alpha,\beta}(Y)=\frac{1}{\beta-\alpha}\log\frac{b-a}{\alpha+\beta}$. From (\ref{e4}) we get, 
	\begin{align*}
	\xi_{\alpha,\beta}^w(X)&=\frac{1}{\beta-\alpha}\log\left[\dfrac{(\beta-\alpha)(a(\alpha+\beta)+b)}{(\alpha+\beta)(\alpha+\beta+1)} \right],\\
	\xi_{\alpha,\beta}^w(Y)&=\frac{1}{\beta-\alpha}\log\left[\dfrac{(\beta-\alpha)(a(\alpha+\beta)+b+h(\alpha+\beta+1))}{(\alpha+\beta)(\alpha+\beta+1)} \right].
	\end{align*}
	
\end{ex}
So we see that, $\xi_{\alpha,\beta}(X)=\xi_{\alpha,\beta}(Y)$ but GWSE of $X$ is smaller than GWSE of $Y$.\\
	
The following lemma shows that $\xi_{\alpha,\beta}^w(X)$ is shift-dependent measure.

\begin{lem}\label{l2}
	Consider the linear transformation $Z=aX+b$, where $a>0$ and $b\geq 0$, then 
	\begin{align}\label{e5}
	\exp[(\beta-\alpha)\xi_{\alpha,\beta}^w(Z)]= a^2\;\exp[(\beta-\alpha)\xi_{\alpha,\beta}^w(X)]+ab\;\exp[(\beta-\alpha)\xi_{\alpha,\beta}(X)]
	\end{align}
\end{lem}
\begin{proof}
	The results follows using $\bar{F}_{aX+b}(x)=\bar{F}_X(\frac{x-b}{a})$, $x\in R$.
\end{proof}

Let $\bar{F}_{X_{\theta}}(x)$ and $\bar{F}(x)$ denote the sfs of the rvs $X_{\theta}$ and $X$, respectively. $X_{\theta}$ and $X$ satisfy proportional hazard rate model i.e $\bar{F}_{X_{\theta}}(x)=[\bar{F}(x)]^{\theta}$, $\theta (>0)$. The following lemma compares the GWSE of $X$, $X_{\theta}$ and $\theta X$. Proofs are omitted.

\begin{lem}\label{2.2}
	The following statements hold:
	\begin{align*}
	&(a)\; \xi_{\alpha,\beta}^w(X_{\theta})=\left( \dfrac{\theta \beta-\theta \alpha-\theta+1}{\beta-\alpha}\right) \xi_{\theta \alpha, \theta \beta-\theta+1}^w(X)\\
	&(b)\; \xi_{\alpha,\beta}^w(X_{\theta})\leq \xi_{\alpha,\beta}^w(X)\leq  \xi_{\alpha,\beta}^w(\theta X),\; if \;\theta>1\\
	&(c)\;  \xi_{\alpha,\beta}^w(X_{\theta})\geq \xi_{\alpha,\beta}^w(X)\geq  \xi_{\alpha,\beta}^w(\theta X),\; if\; 0<\theta<1
	\end{align*}
\end{lem}
We provide GWSE for exponential and Pareto distributions as examples in Table \ref{t1} to verify Lemma \ref{2.2}, where $\gamma=\alpha+\beta-1$.
\begin{table}[h!]
	\centering
	\caption{GWSE for exponential and Pareto distribution} \label{t1}
	\begin{tabular}{c c c c }
		\hline\\
		cdf & $(\beta-\alpha)\xi_{\alpha,\beta}^w(X)$ & $(\beta-\alpha)\xi_{\alpha,\beta}^w(X_{\theta})$ & $(\beta-\alpha)\xi_{\alpha,\beta}^w(\theta X)$ \\[1ex]
		\hline
		$F(x)=1-e^{-(\lambda x)};\;x>0,\lambda>0$ & $-2\text{log}(\lambda\gamma);\;\lambda>0$ & $-2\text{log}(\lambda\theta\gamma);\;\lambda>0$ & $2\text{log}\theta-2\text{log}(\lambda\gamma);\;\lambda>0$ \\
		$F(x)=1-\left( \frac{b}{x}\right)^a;\;x\geq b>0,a>0$ & $\text{log}\frac{b^2}{a\gamma-2};\;a\gamma>2$ & $\text{log}\frac{b^2}{a\theta\gamma-2};\;a\theta\gamma>2$ & $\text{log}\frac{b^2\theta^2}{\theta\gamma-2};\;a\gamma>2$ \\
		\hline
	\end{tabular}
\\[10pt]
\end{table}
 
\begin{de}
Let $X$ be a continuous non-negative rv with sf $\bar{F}(x)$, then the weighted mean residual life (WMRL) of $X$ is given by
\begin{align}\label{e6}
m^*_F(t)=\int_{t}^{\infty}x\frac{\bar{F}(x)}{\bar{F}(t)}dx,\;\;\bar{F}(t)>0.
\end{align}
\end{de}
Note that, $m^*_F(0)=\int_{0}^{\infty}x\bar{F}(x)dx=\frac{1}{2}E(X^2)$. In the following theorem we provide a bound for GWSE in terms of $m^*_F(0)$.
\begin{thm}
	Let $X$ be a continuous non-negative rv having WMRL $m^*_F(t)$ and GWSE $\xi_{\alpha,\beta}^w(X)$, then 
	\begin{align*}
	\xi_{\alpha,\beta}^w(X)\leq \frac{1}{\beta-\alpha}\log m^*_F(0).
	\end{align*}
\end{thm}
\begin{proof}
	Since $x\bar{F}^{\alpha+\beta-1}(x)\leq x\bar{F}(x)$, taking integral on both sides and dividing by $(\beta-\alpha)$ we get the result.
\end{proof}
\section{Generalized dynamic weighted survival entropy of order $(\alpha,\beta)$}
Now we define the dynamic version of GWSE to study the uncertainty in the residual life of a component $X$. Which is the GWSE of the rv $[X-t|X>t],\;t>0$.
\begin{de}
	 Generalized dynamic weighted survival entropy (GDWSE) of order $(\alpha,\beta)$ of a continuous rv $X$ is defined as 
	\begin{align}\label{e7}
	\xi_{\alpha,\beta}^w(X;t)=\frac{1}{\beta-\alpha}\log\int_{t}^{\infty}x\dfrac{\bar{F}^{\alpha+\beta-1}(x)}{\bar{F}^{\alpha+\beta-1}(t)}dx,\; \beta\geq 1,\; \beta-1\;<\alpha<\;\beta.
	\end{align}
\end{de}
Note that, $\xi_{\alpha,\beta}^w(X;0)=\xi_{\alpha,\beta}^w(X)$.
\begin{lem}
		Suppose $Z=aX+b$, where $a>0$ and $b\geq 0$, then 
	\begin{align*}
	\exp[(\beta-\alpha)\xi_{\alpha,\beta}^w(Z;t)]&= a^2\exp\left[ (\beta-\alpha)\xi_{\alpha,\beta}^w\left( X;\frac{t-b}{a}\right)\right] \\
	&+ab\exp\left[ (\beta-\alpha)\xi_{\alpha,\beta}\left( X;\frac{t-b}{a}\right)\right] .
	\end{align*}
\end{lem}
\begin{proof}
The proof is similar to lemma \ref{l2}. 
\end{proof}
\begin{re} 
 If $b=0$, then from Lemma 2.4 we have
\begin{align}\label{e8}
\xi_{\alpha,\beta}^w(Y;t)=\frac{2\log a}{\beta-\alpha}+\xi_{\alpha,\beta}^w\left( X;\frac{t}{a}\right) 
\end{align}
\end{re}
Now we provide a bound for $\xi_{\alpha,\beta}^w(X;t)$ in terms of WMRL.

	\begin{thm}
		Let $X$ be a continuous non-negative rv with WMRL $m^*_F(t)$ and GDWSE $\xi_{\alpha,\beta}^w(X;t)$, then 
		\begin{align*}
		\xi_{\alpha,\beta}^w(X;t)\leq \frac{1}{\beta-\alpha}\log m^*_F(t).
		\end{align*}
	\end{thm}
\begin{proof}
	Since $\dfrac{\bar{F}(x)}{\bar{F}(t)}<1$ for $x>t$, we have $\left( \dfrac{\bar{F}(x)}{\bar{F}(t)}\right)^{\alpha+\beta-1}<\dfrac{\bar{F}(x)}{\bar{F}(t)}$. Taking integral on both sides and dividing by $(\beta-\alpha)$ and then using (\ref{e7}) we get the result.
\end{proof}
To verify Theorem 2.1 and 2.2 we consider exponential and pareto distributions. The results are given in Table \ref{t2}, where $\gamma=\alpha+\beta-1$.
\begin{table}[h!]
	\centering
	\caption{GWSE for exponential and Pareto distribution}\label{t2}
	\begin{tabular}{c c c c c }
		\hline\\
		cdf & $\xi_{\alpha,\beta}^w(X)$ & $m_F^*(0)$ & $(\xi_{\alpha,\beta}^w(X;t)$ & $m_F^*(t)$ \\[1ex]
		\hline
		$F(x)=1-e^{-(\lambda x)};\;x>0,\lambda>0$ & $\frac{2}{\beta-\alpha}\text{log}(\frac{1}{\lambda\gamma});\;\lambda\gamma>1$ & $\frac{1}{\lambda^2}$ & $\frac{1}{\beta-\alpha}\text{log}\left( \frac{1+t\lambda\gamma}{\lambda^2\gamma^2}\right);\;\lambda\gamma>0$ & $\frac{1+t\lambda}{\lambda^2}$ \\
		$F(x)=1-\left( \frac{b}{x}\right)^a;\;x\geq b>0,a>0$ & $\frac{1}{\beta-\alpha}\text{log}\frac{b^2}{a\gamma-2};\;a\gamma>2$ & $\frac{b^2}{a-2};\;a>2$ & $\text{log}\frac{t^2}{a\gamma-2};\;a\gamma>2$ & $\frac{t^2-a}{a-2};\;a>2$ \\
		\hline
	\end{tabular}
\\ [10pt]
\caption*{Where $\gamma=(\alpha+\beta-1)$.}
	
\end{table}

\begin{de}
	 A non-negative continuous rv $X$ is said to be increasing (decreasing) generalized dynamic weighted survival entropy (IGDWSE (DGDWSE)), if $\xi_{\alpha,\beta}^w(X;t)$ is increasing (decreasing) in $t\;(\geq 0)$.
\end{de}

\begin{thm}
      A non-negative continuous rv $X$ is IGDWSE (DGDWSE) if and only if\\ $\lambda_F(t)\geq(\leq) \dfrac{t}{\alpha+\beta-1}exp[-(\beta-\alpha)\xi_{\alpha,\beta}^w(X;t)]$, $\forall t\geq 0$, where $\lambda_F(t)=\frac{f(t)}{\bar{F}(t)}$, is the hazard function.
\end{thm}
\begin{proof}
	We have
	\begin{align}\label{e9}
	(\beta-\alpha)\xi_{\alpha,\beta}^w(X;t)=\text{log}\left[ \int_{t}^{\infty}x\bar{F}^{\alpha+\beta-1}(x)dx\right] -(\alpha+\beta-1)\text{log}\bar{F}(t).
	\end{align}
	Differentiating (\ref{e9}) with respect to t we get,
	\begin{align*}
		(\beta-\alpha)\xi_{\alpha,\beta}'^w(X;t)= (\alpha+\beta-1)\lambda_F(t)- t\dfrac{\bar{F}^{(\alpha+\beta-1)}(t)}{\int_{t}^{\infty}x\bar{F}^{(\alpha+\beta-1)}(x)dx}.
	\end{align*}
	Using (\ref{e7}) we get,
	\begin{align}\label{e10}
	(\beta-\alpha)\xi_{\alpha,\beta}'^w(X;t)= (\alpha+\beta-1)\lambda_F(t)-t\exp[-(\beta-\alpha)\xi_{\alpha,\beta}^w(X;t)]
	\end{align}
	and the result follows from (\ref{e10}).
	
\end{proof}
\begin{de}[Shaked and Shantikumar 2007]
	Let $X$ and $Y$ be two rvs with sfs $\bar{F}(x)$ and $\bar{G}(x)$, respectively. Then $X$ is said to be smaller than $Y$ in the usual stochastic ordering, denoted by $\stackrel{st}{X\leq Y}$, if $\bar{F}(x)\leq \bar{G}(x)$, for all x.
\end{de}
\begin{de}
	$X$ is said to be smaller than $Y$ in generalized weighted survival entropy ordering, denoted by $\stackrel{wse}{X\leq Y}$, if $\xi_{\alpha,\beta}^w(X)\leq \xi_{\alpha,\beta}^w(Y)$.
\end{de}
\begin{thm}
	Let $X$ and $Y$ be two non-negative continuous rvs with sfs $\bar{F}(x)$ and $\bar{G}(x)$, respectively, then $\stackrel{st}{X\leq Y}\implies\stackrel{wse}{X\leq Y}$.
\end{thm}
\begin{proof}
	Proof easily follows using the definition of GWSE.
\end{proof}
\begin{de}[Shaked and Shantikumar (2007)]
	$X$ is said to be smaller than $Y$ in hazard rate ordering, denoted by $\stackrel{hr}{X\leq Y}$, if $\lambda_F(t)\geq \lambda_G(t)$, $\forall t\geq 0$ or equivalently $\frac{\bar{G}(t)}{\bar{F}(t)}$ is increasing in $t$.
\end{de}	
	\begin{de}
		$X$ is said to be smaller than $Y$ in generalized dynamic weighted survival entropy ordering, denoted by $\stackrel{dwse}{X\leq Y}$, if $\xi_{\alpha,\beta}^w(X;t)\leq \xi_{\alpha,\beta}^w(Y;t)$.
	\end{de}	
		\begin{thm}
			Let $X$ and $Y$ be two non-negative continuous rvs with sfs $\bar{F}(x)$ and $\bar{G}(x)$ and hazard rate functions $\lambda_F(t)$ and $\lambda_G(t)$, respectively. If $\stackrel{hr}{X\leq Y}$ then $\stackrel{dwse}{X\leq Y}$.	
		\end{thm}
	\begin{proof}
		Proof follows using the fact that, $\frac{\bar{F}(x)}{\bar{F}(t)}\leq \frac{\bar{G}(x)}{\bar{G}(t)}$ $\forall x\geq t$.
	\end{proof}
\begin{thm}
	Let $X$ and $Y$ be two non-negative continuous rvs and $\stackrel{dwse}{X\leq(\geq) Y}$. Let $Z_1=a_1X$ and $Z_2=a_2Y$, where $a_1,a_2>0$. Then $\stackrel{dwse}{Z_1\leq(\geq)Z_2}$, if $\xi_{\alpha,\beta}^w(X;t)$ is decreasing in $t>0$ and $a_1\leq(\geq)a_2$.
\end{thm}
\begin{proof}
	Suppose $a_1\leq a_2$. Since $\xi_{\alpha,\beta}^w(X;t)$ is decresasing in $t$, we have, $\xi_{\alpha,\beta}^w(X;\frac{t}{a_1})\leq \xi_{\alpha,\beta}^w(X;\frac{t}{a_2})$. Again, $\xi_{\alpha,\beta}^w(X;\frac{t}{a_2})\leq \xi_{\alpha,\beta}^w(Y;\frac{t}{a_2})$ since $\stackrel{dwse}{X\leq Y}$. Combining these two inequalities we have
	$$\xi_{\alpha,\beta}^w(Z_1;t)=\frac{2\text{log}a_1}{\beta-\alpha}+ \xi_{\alpha,\beta}^w(X;\frac{t}{a_1})\leq \frac{2\text{log}a_2}{\beta-\alpha}+ \xi_{\alpha,\beta}^w(Y;\frac{t}{a_2})=\xi_{\alpha,\beta}^w(Z_2;t).$$ Hence the results. Similarly, when $a_1\geq a_2$, it can be easily shown that $\stackrel{dwse}{Z_1\geq Z_2}$.
\end{proof}

	The next theorem shows that, $\xi_{\alpha,\beta}^w(X;t)$ uniquely determines the underlying survival function.
	\begin{thm}
		Let $X$ be a non-negative continuous rv having pdf $f(x)$ and sf $\bar{F}(x)$. Assume that $\xi_{\alpha,\beta}^w(X;t)<\infty;\;t\geq0,\; \beta-1<\alpha<\beta,\; \beta\geq1$. Then $\xi_{\alpha,\beta}^w(X;t)$ uniquely determines the sf of $X$.
	\end{thm}
\begin{proof} From (\ref{e10}) we have
	\begin{align}
	\lambda_F(t)=\frac{1}{\alpha+\beta-1}((\beta-\alpha)\xi_{\alpha,\beta}'^w(X;t)+t\exp[-(\beta-\alpha)\xi_{\alpha,\beta}^w(X;t)]).\label{e11}
	\end{align}
	Now let $X_1$ and $X_2$ be two rvs with sfs $\bar{F}_1(t)$ and $\bar{F}_2(t)$, GDWSEs $\xi_{\alpha,\beta}^w(X_1;t)$ and $\xi_{\alpha,\beta}^w(X_2;t)$ and hazard functions $\lambda_{F_{1}}(t)$ and $\lambda_{F_{2}}(t)$, respectively.\\
	Suppose,
	$$\xi_{\alpha,\beta}^w(X_1;t)=\xi_{\alpha,\beta}^w(X_2;t),$$ then from (\ref{e11}) we get $\lambda_{F_{1}}(t)=\lambda_{F_{2}}(t)$. Since hazard function uniquely determines the survival function of the underlying distribution, we conclude that, $$\bar{F}_1(t)=\bar{F}_2(t)$$.
\end{proof}

\section{Characterization Results Based on GDWSE}
In this section, we obtain some characterization results for Rayleigh distribution based on GDWSE.
\begin{thm}
	The rv $X$ has constant GDWSE if and only if it has a Rayleigh distribution with survival function $\bar{F}(x)=e^{-\lambda x^2};\; x\geq 0$.
	\begin{proof} The if part of the theorem can be easily obtained by using (\ref{e7}). For the only if part let us assume that $\xi_{\alpha,\beta}^w(X;t)=c.$ Differentiating with respect to $t$ we have
		 $$(\alpha+\beta-1)\lambda_F(t)-t\exp[-(\beta-\alpha)\xi_{\alpha,\beta}^w(X;t)]=0.$$ This implies $\lambda_F(t)=\dfrac{e^{(\alpha-\beta)c}}{\alpha+\beta-1}t$, which is the hazard function of a Rayleigh distribution with survival function $\bar{F}(t)=e^{-\lambda t^2};\; t\geq 0$, where $\lambda = \dfrac{e^{(\alpha-\beta)c}}{2(\alpha+\beta-1)}>0$ as $\alpha+\beta>1$.
	\end{proof}
\end{thm}
\begin{thm}
	Let $X$ be a continuous rv with absolutely continuous survival function $\bar{F}$. Then the relation $(\beta-\alpha)\xi_{\alpha,\beta}^w(X;t)=\log m_F^*(t)-\text{log}(\alpha+\beta-1)$ holds if and only if $X$ has a Rayleigh distribution.
	\begin{proof}
		If part of the theorem is straight-forward. For the only if part, assume $(\beta-\alpha)\xi_{\alpha,\beta}^w(X;t)=\text{log}m_F^*(t)-\text{log}(\alpha+\beta-1)$ holds. Differentiating with respect to $t$ we get
		$$(\beta-\alpha)\xi_{\alpha,\beta}'^w(X;t)=\frac{m_F'^*(t)}{m_F^*(t)}.$$
		Then by using (\ref{e10}) we have
		\begin{align}\label{e12}
		 (\alpha+\beta-1)\lambda_F(t)-t\;exp[-(\beta-\alpha)\xi_{\alpha,\beta}^w(X;t)]=\frac{m_F'^*(t)}{m_F^*(t)}.
		 \end{align}
		Note that $m^*_F(t)=\int_{t}^{\infty}x\frac{\bar{F}(x)}{\bar{F}(t)}dx.$ This implies
		\begin{align}\label{e13}
		m_F'^*(t)=\lambda_F(t)m_F^*(t)-t.
		\end{align}
		By putting (\ref{e13}) in (\ref{e12}), after simplification, we get 
		\begin{align}\label{e14}
		\lambda_F(t)m_F^*(t)=t,
		\end{align}
		Combining (\ref{e13}) and (\ref{e14}), we get $m_F'^*(t)=0$, This implies
		\begin{align}\label{e15}
		m_F^*(t)=c,
		\end{align}
		where c is a constant.
		Using (\ref{e14}) and (\ref{e15}) we obtain, $\lambda_F(t)=\frac{t}{c}$, which is the hazard function of a Rayleigh distribution with survival function $\bar{F}(t)=e^{\frac{-t^2}{c}}$.
		
	\end{proof}
\end{thm}
Now we obtaine a characterization result of the first order statistics based on GWSE. Let Let $X_1,X_2,...,X_n$ be a random sample of size $n$ from $F(x)$. Denote the corresponding order statistics as $X_{1:n},X_{2:n},...,X_{n:n}$, where $X_{i:n}(1\leq i\leq n)$ is the $i^{th}$ order statistic. The sf of $X_{1:n}$ is given by, $\bar{F}_{1:n}(x)=\bar{F}^n(x)$ and GWSE of $X_{1:n}$ is obtained as
\begin{align}\label{e16}
\xi_{\alpha,\beta}^w(X_{1:n})&=\frac{1}{\beta-\alpha}\text{log}\int_{0}^{\infty}x\bar{F}^{n(\alpha+\beta-1)}(x)dx\nonumber\\
&=\frac{1}{\beta-\alpha}\text{log}\int_{0}^{1}\dfrac{v^{n(\alpha+\beta-1)}F^{-1}(1-v)}{f(F^{-1}(1-v))}dv.
\end{align}

The following lemma \cite{mirali2017some} is useful to obtain the characterization results.
\begin{lem}
     If $\eta$ is a continuous function on $[0,1]$, such that $\int_{0}^{1}x^n\eta(x)dx=0$, for $n\geq 0$, then $\eta(x)=0$, $\forall x\in[0,1]$.
\end{lem}
\begin{thm}
		Let $X$ and $Y$ be two non-negative continuous rvs having common support $[0,\infty)$ with cdfs $F(x)$ and $G(x)$, respectively. Then $F(x)=G(x)$ if and only if $\xi_{\alpha,\beta}^w(X_{1:n})=\xi_{\alpha,\beta}^w(Y_{1:n})$, $\forall n$.
\end{thm}
\begin{proof}
	 "If" $\xi_{\alpha,\beta}^w(X_{1:n})=\xi_{\alpha,\beta}^w(Y_{1:n})$, then from (\ref{e16}) we have
	\begin{align*}
	\int_{0}^{1}v^{n(\alpha+\beta-1)}\left[ \dfrac{F^{-1}(1-v)}{f(F^{-1}(1-v))}-\dfrac{G^{-1}(1-v)}{g(G^{-1}(1-v))}\right]dv =0.
	\end{align*}
	 Then from Lemma 3.1 we have $\dfrac{F^{-1}(1-v)}{f(F^{-1}(1-v))}=\dfrac{G^{-1}(1-v)}{g(G^{-1}(1-v))}$ for almost all $v\in(0,1)$. This reduces to $F^{-1}(w)\frac{d}{dw}F^{-1}(w)=G^{-1}(w)\frac{d}{dw}G^{-1}(w)$, where $w=1-v$ and $\frac{d}{dw}F^{-1}(w)=\frac{1}{f(F^{-1}(w)}$. Since $X$ and $Y$ have common support $[0,\infty)$, we conclude that $F^{-1}(w)=G^{-1}(w), 0\leq w\leq 1$. Hence the proof.
\end{proof}
\section{Generalized Weighted Failure and Dynamic Failure Entropy of Order $(\alpha,\beta)$}
\begin{de}
		Generalized weighted failure entropy (GWFE) of order $(\alpha,\beta)$ is defined as 
	\begin{align}\label{e17}
	f\xi_{\alpha,\beta}^w(X)=\frac{1}{\beta-\alpha}\log \int_{0}^{\infty}x{F}^{\alpha+\beta-1}(x)dx,\; \beta\geq 1,\; \beta-1\;<\alpha<\;\beta.
	\end{align}
\end{de}
\begin{ex}
		Let $X$ and $Y$ be two rvs with pdfs $f(u)=\frac{1}{a},\;0<u<a$ and $g(u)=\frac{1}{a},\;h<u<a+h,\;h>0$, respectively. From (\ref{e3}) we have, $f\xi_{\alpha,\beta}(X)=f\xi_{\alpha,\beta}(Y)=\frac{1}{\beta-\alpha}\log \frac{a}{\alpha+\beta}$. From (\ref{e17}) we get,
		\begin{align*}
		f\xi_{\alpha,\beta}^w(X)&=\frac{1}{\beta-\alpha}\log\left[\frac{a^2}{\alpha+\beta+1} \right],\\
		f\xi_{\alpha,\beta}^w(Y)&=\frac{1}{\beta-\alpha}\log\left[\dfrac{a(a(\alpha+\beta)+h(\alpha+\beta+1))}{(\alpha+\beta)(\alpha+\beta+1)} \right].
		\end{align*}	
\end{ex}
Although $f\xi_{\alpha,\beta}(X)=f\xi_{\alpha,\beta}(Y)$  but $f\xi_{\alpha,\beta}^w(X)\neq f\xi_{\alpha,\beta}^w(Y)$. The following lemma shows the shift-dependency of GWFE.
\begin{lem}
		Suppose $Z=aX+b$, where $a>0$ and $b\geq 0$, then 
	\begin{align}\label{e18}
	\exp[(\beta-\alpha)f\xi_{\alpha,\beta}^w(Z)]= a^2\exp[(\beta-\alpha)f\xi_{\alpha,\beta}^w(X)]+ab\exp[(\beta-\alpha)f\xi_{\alpha,\beta}(X)].
	\end{align}
\end{lem}
Let ${F}_{X_{\theta}}(x)$ and ${F}(x)$ denote the distribution functions of the rvs $X_{\theta}$ and $X$, respectively, then proportional reverse hazard rate model is described by the relation ${F}_{X_{\theta}}(x)=[{F}(x)]^{\theta}$ $\theta (>0)$. The following lemma compares the GWFE of $X$, $X_{\theta}$ and $\theta X$. Proofs are omitted.

\begin{lem}\label{4.2}
	The following statements hold:
	\begin{align*}
	&(a)\; f\xi_{\alpha,\beta}^w(X_{\theta})=\left( \dfrac{\theta \beta-\theta \alpha-\theta+1}{\beta-\alpha}\right) f\xi_{\theta \alpha, \theta \beta-\theta+1}^w(X).\\
	&(b)\; f\xi_{\alpha,\beta}^w(X_{\theta})\leq f\xi_{\alpha,\beta}^w(X)\leq  f\xi_{\alpha,\beta}^w(\theta X),\; if \theta>1.\\
	&(c)\;  f\xi_{\alpha,\beta}^w(X_{\theta})\geq f\xi_{\alpha,\beta}^w(X)\geq  f\xi_{\alpha,\beta}^w(\theta X),\; if 0<\theta<1.
	\end{align*}
\end{lem}
For illustration we consider exponential and Pareto distributions. Lemma \ref{4.2} can be easily verified by using the results in Table \ref{t3}.
\begin{table}[h!]
	\centering
	\caption{GWFE for uniform and Power distribution where $\gamma=\alpha+\beta-1$.}\label{t3}
	\begin{tabular}{c c c c }
		\hline\\
		cdf & $(\beta-\alpha)f\xi_{\alpha,\beta}^w(X)$ & $(\beta-\alpha)f\xi_{\alpha,\beta}^w(X_{\theta})$ & $(\beta-\alpha)f\xi_{\alpha,\beta}^w(\theta X)$ \\[1ex]
		\hline
		$F(x)=\frac{x}{a};\;0<x<a,\;a>0$ & $\text{log}\frac{a^2}{2+\gamma}$ & $\text{log}\frac{a^2}{2+\gamma\theta}$ & $\text{log}\frac{a^2\theta^2}{2+\gamma}$ \\
		$F(x)=x^c;\;0<x<1;\;c>0$ & $\text{log}\frac{1}{2+\gamma c}$ & $\text{log}\frac{1}{2+\gamma\theta c}$ & $\text{log}\frac{\theta^2}{2+\gamma c}$ \\
		\hline
	\end{tabular}
\\[10pt]

\end{table}

\begin{de}
	Generalized dynamic weighted failure entropy (GDWFE) of order $(\alpha,\beta)$ a rv $X$ is the GWFE of the rv $[t-X|X<t],\;t>0$. GDWFE of $X$ is defined as 
	\begin{align}\label{e19}
	f\xi_{\alpha,\beta}^w(X;t)=\frac{1}{\beta-\alpha}\log\int_{0}^{t}x\dfrac{{F}^{\alpha+\beta-1}(x)}{{F}^{\alpha+\beta-1}(t)}dx,\; \beta\geq 1,\; \beta-1\;<\alpha<\;\beta.
	\end{align}
	Note that, $ f\xi_{\alpha,\beta}^w(X;\infty)=\xi_{\alpha,\beta}^w(X)$.
\end{de}

\begin{lem}\label{l5}
Suppose $Z=aX+b$, where $a>0$ and $b\geq 0$, then 
	\begin{align*}
	\exp[(\beta-\alpha)f\xi_{\alpha,\beta}^w(Y;t)]&= a^2\exp\left[ (\beta-\alpha)f\xi_{\alpha,\beta}^w\left( X;\frac{t-b}{a}\right) \right] \\
	&+ab\exp\left[ (\beta-\alpha)f\xi_{\alpha,\beta}\left( X;\frac{t-b}{a}\right)\right].
	\end{align*}
\end{lem}

\begin{re} 
	If $b=0$, then from Lemma \ref{l5} we have
	\begin{align}\label{e20}
	f\xi_{\alpha,\beta}^w(Y;t)=\frac{2\text{log}a}{\beta-\alpha}+f\xi_{\alpha,\beta}^w\left( X;\frac{t}{a}\right) .
	\end{align}
\end{re}
Now we provide some bounds of GWFE and GDWFE in terms of weighted mean inactivity time (WMIT).
\begin{de}
	The weighted mean inactivity time of a rv $X$ is defined as 
	\begin{align}\label{e21}
	\mu^*_F(t)=\int_{0}^{t}x\frac{{F}(x)}{{F}(t)}dx,\;F(t)>0.
	\end{align}
	Note that $\mu^*_F(\infty)=\lim_{t\to \infty}\mu_F^*(t)= \int_{0}^{\infty}x{F}(x)dx$.
\end{de}
\begin{thm}
	Let $X$ be a non-negative continuous rv with WMIT $\mu_F^*(t)$, GWFE $f\xi_{\alpha,\beta}^w(X)$ and GDWFE $f\xi_{\alpha,\beta}^w(X;t)$. Then
	\begin{align*}
	&(i) f\xi_{\alpha,\beta}^w(X)\leq \frac{1}{\beta-\alpha}\log[\mu_F^*(\infty)].\\
	&(ii) f\xi_{\alpha,\beta}^w(X;t)\leq \frac{1}{\beta-\alpha}\log[\mu_F^*(t)].
	\end{align*} 
\end{thm}
\begin{proof}
	Proofs are similar to that of theorems 2.1 and 2.2.
\end{proof}
\begin{de}
	A non-negative rv $X$ is said to be increasing (decreasing) generalized dynamic weighted failure entropy (IGDWFE (DGDWFE)), if $f\xi_{\alpha,\beta}^w(X;t)$ is increasing (decreasing) in $t\;(\geq0)$.
\end{de}
\begin{thm}
	A non-negative continuous rv $X$ is IGDWFE (DGDWFE) if and only if $$r_F(t)\geq(\leq) \dfrac{t}{\alpha+\beta-1}\exp[-(\beta-\alpha)f\xi_{\alpha,\beta}^w(X;t)],\;\forall t\geq 0,$$  where $r_F(t)=\frac{f(t)}{{F}(t)}$ is the reverse hazard function.
\end{thm}
\begin{proof}
	Differentiating (\ref{e19}) we get,
	\begin{align}\label{e22}
	(\beta-\alpha)f\xi_{\alpha,\beta}'^w(X;t)=t\text{exp}[-(\beta-\alpha)f\xi_{\alpha,\beta}^w(X;t)]-(\alpha+\beta-1)r_F(t).
	\end{align}
	The result follows from (\ref{e22}).
\end{proof}
\begin{de}
	$X$ is said to be smaller than $Y$ in generalized weighted failure entropy ordering, denoted by $\stackrel{wfe}{X\leq Y}$, if $f\xi_{\alpha,\beta}^w(X)\leq f\xi_{\alpha,\beta}^w(Y)$.
\end{de}
\begin{de}[Shaked and Shantikumar (2007)]
	$X$ is said to be smaller than $Y$ in reversed hazard rate ordering, denoted by $\stackrel{rh}{X\leq Y}$, if $r_F(t)\leq r_G(t)$, $\forall t\geq 0$ or equivalently $\frac{{G}(t)}{{F}(t)}$ is increasing in $t$.
\end{de}	
\begin{de}
	$X$ is said to be smaller than $Y$ in generalized dynamic weighted failure entropy ordering, denoted by $\stackrel{dwfe}{X\leq Y}$, if $f\xi_{\alpha,\beta}^w(X;t)\leq f\xi_{\alpha,\beta}^w(Y;t)$, for $t>0$.
\end{de}
\begin{thm}
	Let $X$ and $Y$ be two non-negative continuous rvs with cdfs $F$ and $G$, respectively then $\stackrel{st}{X\leq Y}$ $\implies$ $\stackrel{wfe}{X\geq Y}$.
\end{thm}
\begin{thm}
		Let $X$ and $Y$ be two non-negative continuous rvs with cdfs $F$ and $G$ and reversed hazard functions $r_F(t)$ and $r_G(t)$, respectively then  $\stackrel{hr}{X\leq Y}$ $\implies$  $\stackrel{dwfe}{X\geq Y}$.
\end{thm}		
		\begin{proof}
			Proof follows using the fact if $\stackrel{hr}{X\leq Y}$ then $\frac{F(x)}{F(t)}>\frac{G(x)}{G(t)}$.
		\end{proof} 
\begin{thm}
	Let $X$ and $Y$ be two non-negative continuous rvs and $\stackrel{dwfe}{X\leq(\geq) Y}$. Let $Z_1=a_1X$ and $Z_2=a_2Y$, where $a_1,a_2>0$. Then $\stackrel{dwfe}{Z_1\leq(\geq)Z_2}$, if $\xi_{\alpha,\beta}^w(X;t)$ is decreasing in $t>0$ and $a_1\leq(\geq)a_2$.
\end{thm}
\begin{proof}
	Proof follows along the same line as Theorem 2.6.
\end{proof}
Next theorem shows that GDWFE uniquely determines the distribution function of the underlying distribution.
		\begin{thm}
		Let $X$ be a non-negative continuous rv having pdf $f(x)$ and distribution function ${F}(x)$. Assume that, $f\xi_{\alpha,\beta}^w(X;t)<\infty;\;t\geq0,\;\forall \beta-1<\alpha<\beta,\; \beta\geq1$. Then for each $\alpha$ and $\beta$, $f\xi_{\alpha,\beta}^w(X;t)$ uniquely determines the cdf of $X$.
	\end{thm}
\begin{proof} From (\ref{e22}) we have
	\begin{align}
	r_F()t=\frac{1}{\alpha+\beta-1}(t\text{exp}[-(\beta-\alpha)f\xi_{\alpha,\beta}^w(X;t)]-(\beta-\alpha)f\xi_{\alpha,\beta}'^w(X;t)).\label{e0}
	\end{align}
	Let $F_1(t)$ and $F_2(t)$be two distribution functions with generalized dynamic weighted failure entropies as $f\xi_{\alpha,\beta}^w(X_1;t)$ and $f\xi_{\alpha,\beta}^w(X_2;t)$ and the reverse hazard rate functions $r_{F_1}(t)$ and $r_{F_2}(t)$, respectively. Assume that 
	$f\xi_{\alpha,\beta}^w(X_1;t)=f\xi_{\alpha,\beta}^w(X_2;t)$ holds. Then from (\ref{e0}) we have $r_{F_1(t)}=r_{F_2(t)}$. Since reverse haxzard rate uniquely determines the distribution function of the underlying distribution, we obtain $F_1(t)=F_2(t)$.
\end{proof}
	Now we provide some characterization results for power distribution based on GDWFE.
	\begin{thm}
		Let $X$ be a non-negative rv having support $(0,b)$, with absolutely continuous distribution function $F(x)$ and reversed hazard rate function $r_F(x)$. Then $X$ has a power distribution with $F(x)=\left( \frac{x}{b}\right)^c,\;0<x<b,\;c>0$ if and only if $$(\beta-\alpha)f\xi_{\alpha,\beta}^w(X;t)=\log k+\log\mu_F^*(t),$$
		where $\mu_F^*(t)$ is the WMIT function of $X$ and $k(>0)$ is a constant.
	\end{thm}
\begin{proof}
	If part is straight-forward. Suppose the relation $(\beta-\alpha)f\xi_{\alpha,\beta}^w(X;t)=\text{log}k+\text{log}\mu_F^*(t)$ holds. Differentiating with respect to $t$ we get
	\begin{align}\label{e23}
	t\text{exp}[-(\beta-\alpha)f\xi_{\alpha,\beta}^w(X;t)]-(\alpha+\beta-1)r_F(t)=\frac{\mu_F'^*(t)}{\mu_F^*(t)}.
	\end{align}
	Substituting the value of $(\beta-\alpha)f\xi_{\alpha,\beta}^w(X;t)$ in (\ref{e23}) and using the fact that $\mu_F'^*(t)=\frac{d}{dt}\mu_F(t)=\frac{t-r_F(t)\mu_F^*(t)}{\mu_F^*(t)}$ and after some calculation (\ref{e23}) reduces to
	\begin{align}\label{e24}
	r_F(t)\mu_F^*(t)=\frac{1-k}{k(\alpha+\beta-2)}t.
	\end{align}
	This implies $$\mu_F^*(t)=\frac{k(\alpha+\beta-1)-1}{k(\alpha+\beta-2)}t.$$
	Integrating with respect to $t$ and taking $\mu_F^*(0)=0$ we get,
	$$\mu_F^*(t)=\frac{1-k}{k(\alpha+\beta-2)}\frac{t^2}{2}.$$
	From (\ref{e24}) we obtain $$r_F(t)=\frac{2(1-k)}{k(\alpha+\beta-1)-1}\frac{1}{t}=\frac{c}{t}$$
	where $c=\frac{2(1-k)}{k(\alpha+\beta-1)-1}>0$, for $1>k>\frac{1}{\alpha+\beta-1}$. So we see that $r_F(t)$ is the reverse hazard rate function of the power distribution with distriution function $F(x)=\left( \frac{x}{b}\right)^c,\;0<x<b,\;c>0$. Hence the result. 
\end{proof}
 Next we obtain a characterization result of largest order statistic based on GDWFE. The cdf of $X_{n:n}$ is given by $F_{n:n}(x)=F^n(x)$
and the GDWFE of $X_{n:n}$ is obtained as
\begin{align}\label{e25}
f\xi_{\alpha,\beta}^w(X_{n:n};t)&=\frac{1}{\beta-\alpha}\text{log}\int_{0}^{\infty}xF^{n(\alpha+\beta-1)}(x)dx\nonumber\\
&=\frac{1}{\beta-\alpha}\text{log}\int_{0}^{1}\frac{v^{n(\alpha+\beta-1)}F^{-1}(v)}{f(F^{-1}(v))}dv.
\end{align}
\begin{thm}
	Let $X$ and $Y$ be two non-negative continuous rvs having common support $(0,\infty)$ with cdfs $F(x)$ and $G(x)$, respectively. Then $F(x)=G(x)$
	if and only if $f\xi_{\alpha,\beta}^w(X_{n:n};t)=f\xi_{\alpha,\beta}^w(Y_{n:n};t)$, $\forall n$. 
\end{thm}
\begin{proof}
	The "only if" part is straight forward. For the "if" part assume that,
	$f\xi_{\alpha,\beta}^w(X_{n:n};t)=f\xi_{\alpha,\beta}^w(Y_{n:n};t)$ holds. Now from (\ref{e25}) we have,
	\begin{align*}
	\int_{0}^{1}v^{n(\alpha+\beta-1)}\left[\frac{F^{-1}(v)}{f(F^{-1}(v))}-\frac{G^{-1}(v)}{g(G^{-1}(v))} \right]dv.
	\end{align*}
\end{proof}
 Then from Lemma 4.1 we have, $\frac{F^{-1}(v)}{f(F^{-1}(v))}=\frac{G^{-1}(v)}{g(G^{-1}(v))}$ for almost all $v\in(0,1)$. The rest of the proof is similar to the proof of Theorem 3.3.

\section{Some Inequalities and Bounds} In this section we provide some upper and lower bounds for generalized weighted survival and failure entropies and their dynamic versions.
\begin{thm}
	Let $X$ be a non-negative continuous random variable with pdf $f(x)$, cdf $F(x)$ and sf $\bar{F}(x)$. The following inequalities holds:
	\begin{align*}
	&(i)\; (\beta-\alpha)\xi_{\alpha,\beta}^w(X)+(\alpha+\beta-1)\geq H(X)+E(\text{log}X).\\
	&(ii)\; (\beta-\alpha)f\xi_{\alpha,\beta}^w(X)+(\alpha+\beta-1)\geq H(X)+E(\text{log}X).
	\end{align*}
\end{thm}
\begin{proof}Using log-sum inequality we have
\begin{align}\label{e26}
\int_{0}^{\infty}f(x)\text{log}\frac{f(x)}{x\bar{F}^{\alpha+\beta-1}(x)}dx&\geq \text{log}\dfrac{\int_{0}^{\infty}f(x)dx}{\int_{0}^{\infty}x\bar{F}^{\alpha+\beta-1}(x)dx}\int_{0}^{\infty}f(x)dx\nonumber\\
&=-\text{log}\int_{0}^{\infty}x\bar{F}^{\alpha+\beta-1}(x)dx\nonumber\\
&=-(\beta-\alpha)\xi_{\alpha,\beta}^w(X)
\end{align}
Now the L.H.S of (\ref{e26}) equals
\begin{align*}
\int_{0}^{\infty}(\text{log}f(x))f(x)dx-\int_{0}^{\infty}(\text{log}x)f(x)dx-(\alpha+\beta-1)\int_{0}^{\infty}\text{log}\bar{F}(x)f(x)dx,
\end{align*}
which reduces to $-H(X)-E(\text{log}X)+(\alpha+\beta-1)$. The result follows from (\ref{e26}). Part (ii) follows along the same line as part (i).
\end{proof}
In the next theorem we provide lower bound for GDWSE and GDWFE.
\begin{thm}
	Under the assumptions of Theorem 5.1, the following inequalities holds:
	\begin{align*}
		&(i)\; (\beta-\alpha)\xi_{\alpha,\beta}^w(X;t)+(\alpha+\beta-1)\geq H(X;t)+\int_{t}^{\infty}\frac{f(x)}{\bar{F}(t)}\log(x)\;dx\\
	&(ii)\; (\beta-\alpha)f\xi_{\alpha,\beta}^w(X;t)+(\alpha+\beta-1)\geq \bar{H}(X;t)+\int_{0}^{t}\frac{f(x)}{F(t)}\log(x)\;dx
	\end{align*}
\end{thm}
\begin{proof}
	$(i).$ From log-sum inequality we get
	\begin{align}\label{e27}
	\int_{t}^{\infty}f(x)\text{log}\frac{f(x)}{x \left( \frac{\bar{F}(x)}{\bar{F}(t) }\right) ^{\alpha+\beta-1}}dx&\geq \text{log}\dfrac{\int_{t}^{\infty}f(x)dx}{\int_{t}^{\infty}x\left( \frac{\bar{F}(x)}{\bar{F}(t) }\right) ^{\alpha+\beta-1}dx}\int_{0}^{\infty}f(x)dx\nonumber\\
	&=\bar{F}(t)[\text{log}\bar{F}(t)-(\beta-\alpha)\xi_{\alpha,\beta}^w(X;t)]
	\end{align}
	After some simplifications, R.H.S of (\ref{e27}) reduces to $\int_{t}^{\infty}(\text{log}f(x))f(x)dx-\int_{t}^{\infty}(\text{log}x)f(x)dx+(\alpha+\beta-1)\bar{F}(t)$. Using the definition of $H(X;t)$ and after some simpifications, the results follows from (\ref{e27}). Proof of part $(ii)$ follows similarly.
	
\end{proof}
Now we provide an upper bound for GDWSE and GDWFE.
\begin{thm}
	Under the assumptions of Theorem 5.1 and $X$ having support $[0,b]$, the following inequality holds:
	$$\xi_{\alpha,\beta}^w(X;t)\leq\frac{\int_{t}^{b}x\left( \frac{\bar{F}(x)}{\bar{F}(t)}\right)^{(\alpha+\beta-1)}\text{log}\left[x\left( \frac{\bar{F}(x)}{\bar{F}(t)}\right)^{(\alpha+\beta-1)} \right]dx }{(\beta-\alpha)\int_{t}^{b}x\left( \frac{\bar{F}(x)}{\bar{F}(t)}\right)^{(\alpha+\beta-1)}dx}+\frac{\log(b-t)}{\beta-\alpha},\;t<b.$$
\end{thm}
\begin{proof}
Using log-sum inequality we get,
\begin{align}\label{e28}
\int_{t}^{b}x\left( \frac{\bar{F}(x)}{\bar{F}(t)}\right)&^{(\alpha+\beta-1)}\text{log}\left[x\left( \frac{\bar{F}(x)}{\bar{F}(t)}\right)^{(\alpha+\beta-1)} \right]dx\nonumber\\&\geq \text{log}\dfrac{\int_{t}^{b}x\left( \frac{\bar{F}(x)}{\bar{F}(t)}\right)^{(\alpha+\beta-1)}dx}{b-t}\int_{t}^{b}x\left( \frac{\bar{F}(x)}{\bar{F}(t)}\right)^{(\alpha+\beta-1)}dx\nonumber\\
&=[(\beta-\alpha)\xi_{\alpha,\beta}^w(X;t)-\text{log}(b-t)]\int_{t}^{b}x\left( \frac{\bar{F}(x)}{\bar{F}(t)}\right)^{(\alpha+\beta-1)}dx.
\end{align}
The proof follows from (\ref{e28}).
\end{proof}
\begin{p}
	Under the assumptions of Theorem 5.1, the following inequality holds:
		$$f\xi_{\alpha,\beta}^w(X;t)\leq\frac{\int_{0}^{t}x\left( \frac{{F}(x)}{{F}(t)}\right)^{(\alpha+\beta-1)}\text{log}\left[x\left( \frac{{F}(x)}{{F}(t)}\right)^{(\alpha+\beta-1)} \right]dx }{(\beta-\alpha)\int_{0}^{t}x\left( \frac{{F}(x)}{{F}(t)}\right)^{(\alpha+\beta-1)}dx}+\frac{\text{log}(t)}{\beta-\alpha}.$$
\end{p}
\begin{proof}
	Proof is similar to Theorem 5.3.
\end{proof}
\section{Empirical GWSE and GWFE} Let $X_1,X_2,...,X_n$ be a random sample of size $n$ drawn from a distribution with cdf $F(x)$, sf $\bar{F}(x)$ and $X_{1:n}\leq X_{2:n}\leq...\leq X_{n:n}$ be the corresponding order statistics. Let $F_n(x)$ be the empirical distribution function of $X$ then for $X_{i:n}\leq x<X_{(i+1):n}$
\begin{align*}
F_n(x)=\frac{i}{n};\;i=1,2,\cdots n-1.
\end{align*} 
The emperical GWSE is defined as 
\begin{align}\label{e29}
\hat{\xi}_{\alpha,\beta}^w(X)=\frac{1}{\beta-\alpha}\text{log}\int_{0}^{\infty}x\bar{F_n}^{\alpha+\beta-1}(x)dx,\; \beta\geq 1,\; \beta-1\;<\alpha<\;\beta.
\end{align}
Substituting $\bar{F}_n(x)=1-\frac{i}{n},\;i=1,2,\cdots n-1,)$ in (\ref{e29}) we get, 
\begin{align}\label{e30}
\hat{\xi}_{\alpha,\beta}^w(X)&=\frac{1}{\beta-\alpha}\text{log}\left[ \sum_{i=1}^{n-1}\int_{X_{i:n}}^{X_{(i+1):n}}x\bar{F_n}^{\alpha+\beta-1}(x)dx\right] \nonumber\\
&=\frac{1}{\beta-\alpha}\text{log}\left[ \sum_{i=1}^{n-1}\dfrac{X^2_{(i+1):n}-X^2_{i:n}}{2}\left(1-\frac{i}{n} \right)^{\alpha+\beta-1}\right] \nonumber\\
&=\frac{1}{\beta-\alpha}\text{log}\left[ \frac{1}{2}\sum_{i=1}^{n-1}U_{i+1}\left(1-\frac{i}{n} \right)^{\alpha+\beta-1}\right],
\end{align}
where $U_{i+1}=\dfrac{X^2_{(i+1):n}-X^2_{i:n}}{2}$ and $U_1=X_{1:n}$.
Similarly, emperical GWFE can be obtained as 
\begin{align}\label{e31}
\hat{f\xi}_{\alpha,\beta}^w(X)=\frac{1}{\beta-\alpha}\text{log}\left[ \frac{1}{2}\sum_{i=1}^{n-1}U_{i+1}\left(\frac{i}{n} \right)^{\alpha+\beta-1}\right].
\end{align}

\section{Application} In this section we consider the difference between $\xi_{\alpha,\beta}^w(X)$ and its empirical version $\hat{\xi}_{\alpha,\beta}^w(X)$ as a test statistic for testing exponentiality. Let $X_1,X_2,\cdots,X_n$ be iid rvs from a non-negative absolutely continuous cdf $F$. Let $F_0(x,\lambda)=1-e^{-\lambda x},\;x>0,\;\lambda>0$, denote the cdf of a exponential distribution with parameter $\lambda$. We want to test the hypothesis 
\begin{align*}
H_0:F(x)=F_0(x,\lambda)\;\;\;\;vs.\;\;\;\;H_1:F(x)\neq F_0(x,\lambda).
\end{align*} 
Now consider the absolute difference between $\xi_{\alpha,\beta}^w(X)$ and $\hat{\xi}_{\alpha,\beta}^w(X)$ as $D=\mid \xi_{\alpha,\beta}^w(X)-\hat{\xi}_{\alpha,\beta}^w(X)\mid $. If $X\sim \exp(\lambda)$ then $\xi_{\alpha,\beta}^w(X)=\frac{2}{\alpha-\beta}\text{log}(\lambda(\alpha+\beta-1))$ and $D$ reduces to $D=\mid\hat{\xi}_{\alpha,\beta}^w(X)-\frac{2}{\alpha-\beta}\text{log}(\hat{\lambda}(\alpha+\beta-1))\mid$, where $\hat{\lambda}=1/\bar{X}$ is the maximum likelihood estimatir (mle) of $\lambda$. $D$ measures the distance between GWSE and empirical GWSE and large values of $D$ indicates that the sample is from a non-exponential family. Now consider the monotone transformation $T=\exp(-D)$, where $0<T<1$. Under the null hypothesis, $D\overset{p}{\to}0$ and hence $T\overset{p}{\to}1$. So we reject $H_0$ at the significance level $\gamma$ if $T<T_{\gamma,n}$, where $T<T_{\gamma,n}$ is the lower $\gamma$-quantile of the edf of $T$.\\

The sampling distribution of $T$ under $H_0$ is intractable. So to obtain the critical points $T_{\gamma,n}$ by simulations we generate 10000 samples of size $n$ from a standard exponential distribution has been generated for $n=1(1)30,\; 30(5)50$ and $50(10)100$. For each $n$ the lower $\gamma$-quantile of the edf of $T$ is used to determine $T_{\gamma,n}$. The critical points varies for different choice of $(\alpha,\beta)$. The critical points of 90\%, 95\% and 99\% are presented in the table \ref{t4} for $\alpha=0.26$ and $\beta=1.25$.

\begin{table}[h!]
	\centering
	\caption{Critical values of T}\label{t4}
	\begin{tabular}{c c c c c c c c }
		\hline\\
		$n$ & $T_{0.01,n}$ & $T_{0.05,n}$ & $T_{0.10,n}$& $n$ & $T_{0.01,n}$& $T_{0.05,n}$&$T_{0.10,n}$\\[1ex]
		\hline
		 4 & 0.07666 & 0.13727  & 0.16981 & 22 & 0.30434 & 0.35431 & 0.38425 \\
		 5 & 0.12145 & 0.17263 & 0.20185  &  23 & 0.30909 & 0.35696 & 0.39035\\
		 6 & 0.14906 & 0.19960 & 0.22662 &  24 & 0.31098 & 0.36215 & 0.39297\\
		 7 & 0.16893 & 0.21810 & 0.24337 &  25 & 0.31807 & 0.21810 & 0.37008\\
		 8 & 0.19024 & 0.23468 & 0.26113 & 26 & 0.32124 & 0.37164 & 0.40463\\
		 9 & 0.20317 & 0.24861 & 0.27558 & 27 & 0.32314 & 0.37540 & 0.40856\\
		  10 & 0.21687 & 0.26026  & 0.28849 &  28 & 0.32846 & 0.37966 & 0.41360\\
		 11 & 0.22709 & 0.27192 & 0.30025 &  29 & 0.33505 & 0.38617  & 0.41831\\
		 12 & 0.23581 & 0.28305 & 0.31134 &  30 & 0.34131 & 0.38831 & 0.42192\\
		 13 & 0.24587 & 0.28878 & 0.31884  & 35 & 0.35224 & 0.40191 & 0.43525\\
		 14 & 0.25436 & 0.29767 & 0.32693 & 40 & 0.37268 & 0.42064 & 0.45580 \\
		 15 & 0.26067 & 0.30454 & 0.33363 &   45 & 0.38505 & 0.43506 & 0.47251\\
		  16 & 0.26468 & 0.31550  & 0.34597 &  50 & 0.39104 & 0.44952 & 0.48579\\
		 17 & 0.27219 & 0.32203 & 0.35480 & 60 & 0.41437 & 0.47008 & 0.50450\\
		 18 & 0.28200 & 0.33075 & 0.36135  &  70 & 0.43474 & 0.48912 & 0.52189\\
		 19 & 0.28674 & 0.33222 & 0.36592 &  80 & 0.44990 & 0.50518 & 0.54231\\
		 20 & 0.28955 & 0.33686 & 0.36894 & 90 & 0.46394 & 0.51954 & 0.55417\\
		 21 & 0.29881 & 0.34701 & 0.37865 & 100 & 0.47300 & 0.52891 & 0.56268\\
		  
		\hline
	\end{tabular}
	
\end{table}
 We computhe the power of the test for Weibull and Gamma alternative. We observe through simulation, that the power of test does not changes significantly for different choices of the scale parameters of the alternative distributions. So we take the scale parameters to be 1 in both cases.\\

We calculated the powers of the test based on 10000 samples of size $n=5(5)30$. We obtained the powers for significance level $\gamma=0.01$, $\gamma=0.05$ and $\gamma=0.10$. From table \ref{t5} and \ref{t6} we see that as sample size increases the power of the test also increases, as expected. Also the power We calculated the powers of the tests based on 10000 samples of size $n=10(5)25$. We obtained the powers for significance level $\alpha=0.01$ and $\alpha=0.05$. For power computation we consider two alternative distributions Weibull (p,1) with pdf $f_W(x)=px^{p-1}e^{-x^p},\;x,p>0$ and Gamma (q,1) with pdf $f_{GA}(x)=\frac{e^{-x}x^{q-1}}{\Gamma(q)},\;x,q>0$. The powers for Weibull and gamma alternatives are proposed in tables \ref{t5} and \ref{t6}, respectively. It is observed that the powers of the test $T$ is higher than that of $T^{*}$ but slightly lower than that of $KL_{mn}$ for small sample size $n$ = 10. However, for moderate to large sample sizes the proposed test $T$ behaves similar to $KL_{mn}$ and $T^{*}$.  when the shape parameters increases in both cases. The power is very high even for small sample sizes. So this test can be used as a goodness of fit test for exponential distribution.

 \begin{table}[]
	\begin{minipage}{0.40\textwidth}
		\centering
		\caption{\textbf{Power of the test when the alternative is Weibull(p,1)}}\label{t5}
		\begin{tabular}{c c c c c }
			\hline\\
			$n$ & p & $\gamma=0.01$ & $\gamma=0.05$ & $\gamma=0.10$\\[1ex] 
			\hline
			5 & 2 & 0.1243 & 0.3830 & 0.5597\\
			& 3 & 0.3748 & 0.7520 & 0.8843\\
			& 4 & 0.6376 & 0.9300 & 0.9823\\\hline
			10 & 2 & 0.3819 & 0.6981 & 0.8237\\
			& 3 & 0.8936 & 0.9859 & 0.9962\\
			& 4 & 0.9948 & 1 & 1\\\hline
			15 & 2 & 0.5962 & 0.8305 & 0.9254\\
			& 3 & 0.9896 & 0.9997 & 1\\
			& 4 & 1 & 1 & 1\\\hline
			20 & 2 & 0.7348 & 0.9204 & 0.9680\\
			& 3 & 0.9987 & 1 & 1\\
			& 4 & 1 & 1 & 1\\\hline
			25 & 2 & 0.8552 & 0.9686 & 0.9890\\
			& 3 & 1 & 1 & 1\\
			& 4 & 1 & 1 & 1\\\hline
			30 & 2 & 0.9231 & 0.9837 & 0.9953\\
			& 3 & 1 & 1 & 1\\
			& 4 & 1 & 1 & 1\\\hline
		\end{tabular}

	\end{minipage}
	\hfill
	\begin{minipage}{0.40\textwidth}
		\centering
		\caption{\textbf{Power of the test when the alternative is Gamma(q,1)}}\label{t6}
		\begin{tabular}{c c c c c }
			\hline\\
			$n$ & p & $\gamma=0.01$ & $\gamma=0.05$ & $\gamma=0.10$\\[1ex] 
			\hline
			5 & 5 & 0.2316 & 0.5855 & 0.7674\\
			& 6 & 0.3101 & 0.6913 & 0.8434\\
			& 7 & 0.3745 & 0.7731 & 0.8992\\\hline
			10 & 5 & 0.6516 & 0.8746 & 0.9392\\
			& 6 & 0.7787 & 0.9377 & 0.9699\\
			& 7 & 0.8625 & 0.9712 & 0.9870\\\hline
			15 & 5 & 0.8233 & 0.9452 & 0.9724\\
			& 6 & 0.9206 & 0.9796 & 0.9918\\
			& 7 & 0.9591 & 0.9930 & 0.9968\\\hline
			20 & 5 & 0.9046 & 0.9739 & 0.9900\\
			& 6 & 0.9608 & 0.9917 & 0.9973\\
			& 7 & 0.9869 & 0.9984 & 0.9998\\\hline
			25 & 5 & 0.9556 & 0.9909 & 0.9960\\
			& 6 & 0.9869 & 0.9983 & 0.9990\\
			& 7 & 0.9961 & 0.9997 & 1\\\hline
			30 & 5 & 0.9771 & 0.9928 & 0.9986\\
			& 6 & 0.9943 & 0.9990 & 0.9996\\
			& 7 & 0.9986 & 1 & 1\\\hline
		\end{tabular}
	\end{minipage}
\end{table}

\section{Conclusion} Generalized weighted survival and failure entropies and their dynamic versions are considered. We provide several properties of the said measures and obtained characterizations results for Rayleigh and power distributions based on the dynamic versions. We also provide the empirical versions of the entropy measures and using the difference between GWSE and its empirical version we perform test of exponentiality. The test depends on the choice of $\alpha$ and $\beta$. Optimal choice of $\alpha$ and $\beta$ is an important issue. One can choose $(\alpha,\beta)$ in such a way that the asymptotic variances of the empirical generalized weighted survival and failure entropies are minimum. More work will be needed in this direction.

\bibliographystyle{apacite}
\bibliography{References}
\end{document}